\documentclass[11pt]{amsart}
\usepackage{hyperref}
\usepackage{xypic}
\usepackage{latexsym}
\usepackage{amssymb}
\theoremstyle{plain} % text will be italics
\newtheorem{theorem}{Theorem}[section]
\newtheorem{proposition}[theorem]{Proposition}

\newtheorem{lemma}[theorem]{Lemma}
\newtheorem{conjecture}[theorem]{Conjecture}

\theoremstyle{definition} % text will be roman

\newtheorem{example}[theorem]{Example}

\newtheorem{remark}[theorem]{Remark}
\newtheorem{definition}[theorem]{Definition}

\DeclareMathOperator{\SC}{\mathtt{co}}

\DeclareMathOperator{\RR}{\mathcal{R}}
\DeclareMathOperator{\Msp}{\mathcal{M}}
\DeclareMathOperator{\Wt}{wt}
\DeclareMathOperator{\prim}{prim}
\DeclareMathOperator{\IC}{\mathcal{IC}}

\DeclareMathOperator{\Aut}{Aut}

\DeclareMathOperator{\Sp}{nilp}

\DeclareMathOperator{\KK}{\mathrm{K_0}}
\DeclareMathOperator{\tr}{\mathop{tr}}
\DeclareMathOperator{\Gl}{GL}

\newcommand{\Ho}{\mathrm{H}}

\newcommand{\nilp}{{\rm nilp}}

\renewcommand{\phi}{\varphi}

\newcommand{\Hom}{\mathop{\rm Hom}\nolimits}

\newcommand{\End}{\mathop{\rm End}\nolimits}

\newcommand{\Gr}{\mathop{\rm Gr}\nolimits}

\newcommand{\Sym}{{\mathbf{Sym}}}
\newcommand{\KSym}{\KK({\mathbf{Sym}})}
\newcommand{\KSymp}{\KK({\mathbf{Sym}})_{\pm}}
\newcommand{\KSymm}{\KK({\mathbf{Sym}})_-}

\newcommand{\pt}{\mathop{\rm pt}}

\renewcommand{\tilde}{\widetilde}
\newcommand{\Cp}{\mathbb{C}}

\renewcommand{\AA}{\mathbb A}

\begin{document}

\title[Purity of critical cohomology and Kac's conjecture]{Purity of critical cohomology and Kac's conjecture}
\author[B. Davison ]{Ben Davison}
\address{B. Davison: Department of Mathematics and Statistics\\Glasgow University}
\email{ben.davison@glasgow.ac.uk}

\begin{abstract} 
We provide a new proof of the Kac positivity conjecture for an arbitrary quiver $Q$.  The ingredients are the cohomological integrality theorem in Donaldson--Thomas theory, dimensional reduction, and an easy purity result.  These facts imply the purity of the cohomological Donaldson--Thomas invariants for partially nilpotent representations of a quiver with potential $(\tilde{Q},W)$ associated to $Q$, which in turn implies positivity of the Kac polynomials for $Q$.
\end{abstract}
\maketitle
\tableofcontents
\thispagestyle{empty}

\section{Introduction}
Let $Q$ be a finite quiver.  In \cite{kac83}, Victor Kac introduced the polynomials $a_{\gamma}(q)$, counting the number of isomorphism classes of absolutely indecomposable $Q$-representations of a given dimension vector $\gamma\in\mathbb{N}^{Q_0}$ over $\mathbb{F}_q$, and conjectured that the coefficients of $a_{\gamma}(q)$ are always positive.  This conjecture was proved thirty years later, by Hausel, Letellier and Rodriguez--Villegas in \cite{HLRV13}, via a detailed study of the arithmetic harmonic analysis of Nakajima quiver varieties.  The purpose of this paper is to give a new proof, demonstrating the positivity conjecture as a natural consequence of the cohomological Donaldson--Thomas theory of a quiver with potential $(\tilde{Q},W)$ depending on $Q$.  

We show that via the cohomological integrality theorem of \cite{DaMe15b}, positivity is implied by the purity of a certain $\mathbb{Z}_{\geq 0}^{Q_0}$-graded mixed Hodge structure $\mathcal{H}^{\Sp}_{\tilde{Q},W}$, the underlying mixed Hodge structure of a version of the critical cohomological Hall algebra of Kontsevich and Soibelman \cite[Sec.7]{COHA}.  Here, purity is the property that the associated graded object of the weight filtration for the $j$th cohomologically graded piece of $\mathcal{H}^{\Sp}_{\tilde{Q},W}$ is concentrated entirely in the $j$th place.  Taking the appropriate characteristic function of $\mathcal{H}^{\Sp}_{\tilde{Q},W}$, we recover the refined Donaldson--Thomas partition function for representations of the Jacobi algebra of the pair $(\tilde{Q},W)$ satisfying a nilpotency condition defined in Section \ref{CDTsec}.  It turns out to be quite straightforward to prove the purity of $\mathcal{H}^{\Sp}_{\tilde{Q},W}$, recovering the positivity theorem for the Kac polynomials $a_{\gamma}(q)$.

The purity argument goes via the cohomological version \cite[Thm.A.1]{Chicago2} of the dimensional reduction theorem \cite[Thm.B.1]{BBS}; for an arbitrary dimension vector $\gamma\in\mathbb{N}^{Q_0}$ we use dimensional reduction to prove that the degree $\gamma$ piece of the mixed Hodge structure $\mathcal{H}^{\Sp}_{\tilde{Q},W}$ is isomorphic to the mixed Hodge structure on the compactly supported cohomology of the stack $Z_{\gamma}$ of pairs $(\rho,\lambda)$, where $\rho$ is a $\gamma$-dimensional representation of $Q$ and $\lambda$ is a nilpotent endomorphism of $\rho$.  We give a stratification of this stack $Z_{\gamma}^0\subset\ldots\subset Z_{\gamma}^n=Z_{\gamma}$ such that each $Z_{\gamma}^m\setminus Z_{\gamma}^{m-1}$ has pure compactly supported cohomology and is open in $Z^m_{\gamma}$, from which it follows from the associated long exact sequences in compactly supported cohomology that $Z_{\gamma}$ has pure compactly supported cohomology too.

The idea of relating polynomials arising in positivity conjectures to weight polynomials of mixed Hodge structures from Donaldson--Thomas theory is utilised also in \cite{Efi11}, where the coefficients arising in quantum cluster mutation are related to the coefficients in a particular Donaldson--Thomas (henceforth DT) partition function.  This connection was exploited in \cite{DMSS13} to prove the quantum cluster positivity conjecture for quivers admitting a nondegenerate quasihomogeneous potential.  There is, however, an important difference between that work and this, which is explained in terms of the difference between DT/BPS invariants and the coefficients of DT partition functions, which we now explain.  

Given a quiver with potential $(Q',W'\in\mathbb{C}Q/[\mathbb{C}Q,\mathbb{C}Q]_{\textrm{vect}})$ we consider the DT partition function
\begin{align}
\label{Pfunction}
\mathcal{Z}_{Q',W'}(x,q):=&\chi_{q,x}(\mathcal{H}_{Q',W'})\\ =&\sum_{\gamma\in\mathbb{Z}_{\geq 0}^{Q_0}}\chi_q(\Ho_{c,G_{\gamma}}(M_{Q',\gamma},\phi_{\tr(W')_{\gamma}})^{\vee})x^{\gamma}(-q^{1/2})^{-\chi(\gamma,\gamma)}\in\mathbb{Z}((q^{-1/2}))[[x_i|i\in Q_0]],\nonumber
\end{align}
where the superscript $\vee$ denotes the dual in the category of mixed Hodge structures, for a complex of mixed Hodge structures $L$ one defines
\[
\chi_q(L):=\sum_{j\in\mathbb{Z}}\sum_{n\in\mathbb{Z}}(-1)^j\dim(\Gr^{\Wt}_n(\Ho^j(L)))q^{n/2},
\]
and the other constituent terms of (\ref{Pfunction}) are introduced below.  Note that purity of $L$ implies that $\chi_q(L)$ is a Laurent power series in $-q^{1/2}$ with positive coefficients.  The refined integrality conjecture of Kontsevich and Soibelman \cite{KS} for quivers with potential, proved by Kontsevich and Soibelman in \cite{COHA} (see also \cite{DaMe4}) which is a $q$-analogue of the integrality conjecture of Joyce and Song \cite{JS08}, states that there is a product expansion
\[
\mathcal{Z}_{Q',W'}(x,q):=\prod_{\gamma\in\mathbb{Z}_{\geq 0}^{Q_0}\setminus \{0\}}(1-x^{\gamma})^{\Omega_{\gamma}(q^{-1/2})q^{1/2}/(1-q)}.
\]
Here the $\Omega_{\gamma}(q^{1/2})$ are Laurent polynomials in $q^{1/2}$, and the order in which we take the product (for general, nonsymmetric $Q'$, the $x^{\gamma}$ do not commute), as well as the the polynomials $\Omega_{\gamma}(q^{1/2})$, are determined by a Bridgeland stability condition, even though the spaces $M_{Q',\gamma}$ are not.  The Laurent polynomials $\Omega_{\gamma}(q^{1/2})$ are by definition the refined DT/BPS invariants\footnote{The substitution $q^{1/2}\mapsto q^{-1/2}$ occurs because the refined DT invariants of \cite{KS} are defined with respect to compactly supported cohomology and not dual compactly supported cohomology.} of \cite{KS}.  One should note that although the positivity of the coefficients of $-\Omega_{\gamma}(q^{1/2})$ implies positivity of the coefficients in the partition function $\mathcal{Z}_{Q',W'}(x,q)$, the implication is strictly one way --- this comes down to a multivariable version of the observation that writing a formal power series $f(x)\in\mathbb{Z}[x]$ in terms of an Euler product expansion
\[
f(x)=1+f_1x+f_2x^2+\ldots=\prod_{n\geq 1}(1-x^n)^{-c_n}
\]
the positivity of all the $c_n$ implies the positivity of all the $f_n$, but not vice-versa.

Writing 
\begin{equation}
\label{nilpo}
\chi_{q,x}(\mathcal{H}^{\Sp}_{\tilde{Q},W})=\prod_{\gamma\in\mathbb{Z}_{\geq 0}^{Q_0}\setminus \{0\}}(1-x^{\gamma})^{\Omega^{\nilp}_{\gamma}(q^{-1/2})q^{1/2}/(1-q)}, 
\end{equation}
we will see in Section \ref{ksection} (see (\ref{fkac})) that the Kac polynomials $a_{\gamma}(q)$ are equal, up to a factor of $-q^{1/2}$, to the refined DT invariants $\Omega^{\nilp}_{\gamma}(q^{-1/2})$, so in particular the refined integrality theorem can be demonstrated directly for this partition function.  

In contrast with the cluster algebra setting, in which we are interested only in the positivity of the coefficients of a partition function, proving positivity for the coefficients of $\chi_{q,x}(\mathcal{H}^{\Sp}_{\tilde{Q},W})$ via purity of the mixed Hodge structure $\mathcal{H}^{\Sp}_{\tilde{Q},W}$ does not immediately imply positivity of the Kac polynomials, by the observation above.  The extra ingredient needed is the cohomological upgrade of the refined integrality theorem \cite[Thm.A]{DaMe15b}.  This states that there is a sub \\$(\mathbb{Z}_{\geq 0}^{Q_0}\setminus\{0\})$-graded mixed Hodge structure $\mathfrak{g}^{\nilp}_{\prim}\otimes\Ho_{\mathbb{C}^*}(\pt)=\mathfrak{g}^{\nilp}_{\prim}[u]\subset \mathcal{H}^{\Sp}_{\tilde{Q},W}$, where $u$ is a pure cohomological degree 2 generator, and $(\mathfrak{g}^{\nilp}_{\prim}[u])_{\gamma}:=\mathfrak{g}^{\nilp}_{\prim,\gamma}[u]$ for $\gamma\in(\mathbb{Z}_{\geq 0}^{Q_0}\setminus\{0\})$, and an isomorphism
\begin{equation}
\label{PBWiso}
\mathcal{H}^{\Sp}_{\tilde{Q},W}\cong\Sym(\mathfrak{g}^{\nilp}_{\prim}[u]).
\end{equation}
Furthermore, each graded piece $\mathfrak{g}_{\prim,\gamma}^{\nilp}$ is finite-dimensional.  The left hand side of (\ref{PBWiso}) is the underlying $\mathbb{Z}_{\geq 0}^{Q_0}$-graded mixed Hodge structure of the free unital supercommutative algebra generated by $\mathfrak{g}^{\nilp}_{\prim}[u]$ in the category of $\mathbb{Z}^{Q_0}_{\geq 0}$-graded mixed Hodge structures.  The cohomological DT invariants are then defined (up to a Tate twist) to be the $(\mathbb{Z}^{Q_0}_{\geq 0}\setminus\{0\})$-graded pieces of $\mathfrak{g}^{\nilp}_{\prim}$, and the invariants $\Omega^{\nilp}_{\gamma}(q^{1/2})$, along with the Kac polynomials $a_{\gamma}(q)$, then acquire a new interpretation, as the shifted weight polynomials 
\begin{equation}
\label{shiftexp}
\Omega^{\nilp}_{\gamma}(q^{-1/2})=-q^{-1/2}\chi_q(\mathfrak{g}^{\nilp}_{\prim,\gamma})=-q^{1/2}a_{\gamma}(q^{-1}).  
\end{equation}
Putting the $-q^{1/2}$ factors together, $qa_{\gamma}(q^{-1})=\chi_q(\mathfrak{g}^{\nilp}_{\prim,\gamma})$.  Purity of $\mathcal{H}^{\Sp}_{\tilde{Q},W}$, implies purity of the subobject $\mathfrak{g}_{\prim,\gamma}^{\nilp}\subset \mathcal{H}^{\Sp}_{\tilde{Q},W}$, and so positivity of the coefficients of $\chi_q(\mathfrak{g}^{\nilp}_{\prim,\gamma})$ and $a_{\gamma}(q)$.

\removelastskip\vskip.5\baselineskip\par\noindent{\bf Acknowledgements.}  I would like to thank Sven Meinhardt for originally bringing my attention to this problem and the link to quivers with potential, and for working with me on the foundational results that lie in the background of this paper.  I was supported by the SFB/TR 45 ``Periods, Moduli Spaces and Arithmetic of Algebraic Varieties'' of the DFG (German Research Foundation) during the research contained in this paper, and I would like to thank Daniel Huybrechts for getting me to Bonn.  While completing the writing of the paper, and undertaking corrections, I was employed at the EPFL, supported by the Advanced Grant ``Arithmetic and physics of Higgs moduli spaces'' No. 320593 of the European Research Council.  I would also like to thank Northwestern University for providing excellent working conditions during the writing of this paper, and Ezra Getzler and Kevin Costello for very helpful suggestions while I was there.  This work was partly supported by the NSF RTG grant DMS-0636646.  I would also like to thank the anonymous referee for numerous very helpful suggestions regarding the presentation of what follows.

\section{Cohomological Donaldson--Thomas theory}
\label{CDTsec}
Let $Q$ be a quiver, i.e. the data of two sets $Q_0$ and $Q_1$, the vertices and arrows respectively, and two maps $s,t\colon Q_1\rightarrow Q_0$ taking an arrow to its source and target, respectively.  We assume throughout that $Q_0$ and $Q_1$ are finite.  We define the Ringel form $\chi(-,-)$ by setting 
\[
\chi(\gamma,\gamma')=\sum_{i\in Q_0}\gamma(i)\gamma'(i)-\sum_{a\in Q_1}\gamma(s(a))\gamma'(t(a))
\]
for dimension vectors $\gamma,\gamma'\in\mathbb{Z}_{\geq 0}^{Q_0}$.  We define 
\[
M_{Q,\gamma}=\prod_{a\in Q_1}\Hom(\Cp^{\gamma(s(a))},\Cp^{\gamma(t(a))}), 
\]
and 
\[
G_{\gamma}=\prod_{i\in Q_0}\Gl_{\Cp}(\gamma(i)), 
\]
which acts on $M_{Q,\gamma}$ via change of basis.  
As in \cite{Moz11}, from $Q$ we build a new quiver $\tilde{Q}$ with superpotential $W\in\Cp\tilde{Q}/[\Cp\tilde{Q},\Cp\tilde{Q}]$ by the following procedure:
\begin{enumerate}
\item
For every arrow $a\in Q_1$ we add a new arrow $\tilde{a}$ with $s(\tilde{a})=t(a)$ and $t(\tilde{a})=s(a)$.
\item
For every vertex $i\in Q_0$ we add a new loop $\omega_i$ with $s(\omega_i)=t(\omega_i)=i$.
\item
We set $W=(\sum_{i\in Q_0}\omega_i)(\sum_{a\in Q_1}[a,\tilde{a}])$.
\end{enumerate}
We denote by $\tilde{\chi}$ the Ringel form associated to $\tilde{Q}$. 

%Now define a new quiver with superpotential $(\tilde{Q},W)^T$ by reversing all the arrows of $\tilde{Q}$, and reversing the order of all the cyclic words in $W$.  Define the isomorphism of quivers $\rho:Q\rightarrow Q^T$ by sending $a$ to $\tilde{a}^T$, $\tilde{a}$ to $a^T$, and $\omega_i$ to $\omega_i^T$.  
%\begin{proposition}
%\label{sduality}
%There is an equality $\rho^*(W^T)=-W$, and so in the terminology of \cite{Chicago2}, $(\tilde{Q},W)$ has a self-duality structure such that all dimension vectors are self-dual.
%\end{proposition}

\begin{definition}
Define $M^{\Sp}_{\tilde{Q},\gamma}\subset M_{\tilde{Q},\gamma}$ to be the subvariety of representations which send each $\omega_i$ to a nilpotent endomorphism of $\Cp^{\gamma(i)}$.
\end{definition}  
\begin{definition}
We define a \textit{cut} of $(\tilde{Q},W)$ to be a set of edges $S\subset \tilde{Q}_1$ such that every term of $W$ contains exactly one element of $S$.  Given a cut $S$, we form a quiver $Q_S$ by deleting all of the arrows of $S$ from $\tilde{Q}$, and define the two-sided ideal $I_S\subset \Cp Q_S$ by $I_S:=\langle \frac{\partial W}{\partial a}|a\in S\rangle$, where $\frac{\partial W}{\partial a}$ is defined by cyclically permuting each instance of $a$ in a word of $W$ to the front and then deleting it.  We will always make the assumption that $S$ contains none of the arrows $\omega_i$.
\end{definition}
Given a cut $S$, we define 
\[
Z_{\tilde{Q},W,S,\gamma}\subset M_{Q_S,\gamma}
\]
to be the closed subvariety of representations of $Q_S$ satisfying the relations defined by $I_S$, and define
\[
\overline{Z}_{\tilde{Q},W,S,\gamma}:=Z_{\tilde{Q},W,S,\gamma}\times\prod_{a\in S}\Hom(\Cp^{\gamma(s(a))},\Cp^{\gamma(t(a))})\subset M_{\tilde{Q},\gamma}.  
\]
We define the $G_{\gamma}$-equivariant subvariety $Z^{\nilp}_{\tilde{Q},W,S,\gamma}\subset Z_{\tilde{Q},W,S,\gamma}$ by the condition that each of the loops $\omega_i$ is sent to a nilpotent linear map, and define 
\[
\overline{Z}_{\tilde{Q},W,S,\gamma}^{\nilp}:=Z^{\nilp}_{\tilde{Q},W,S,\gamma}\times\prod_{a\in S}\Hom(\Cp^{\gamma(s(a))},\Cp^{\gamma(t(a))}).  
\]

The following is trivial, and holds precisely because we have chosen $S$ not to contain any of the loops $\omega_i$.  It is required for Theorem \ref{dimred}.
\begin{proposition}
\label{stableunderpi}
Let 
\begin{equation}
\label{pidef}
\pi\colon M_{\tilde{Q},\gamma}\rightarrow M_{Q_S,\gamma}
\end{equation}
be the natural projection, then 
\[
M^{\Sp}_{\tilde{Q},\gamma}=\pi^{-1}\pi(M^{\Sp}_{\tilde{Q},\gamma}).
\]
\end{proposition}
Next define 
\[
\mathcal{H}_{\tilde{Q},W,\gamma}=\left(\Ho_{c,G_{\gamma}}\left(M_{\tilde{Q},\gamma},\phi_{\tr(W)_{\gamma}}\right)\otimes\mathbb{Q}(-\tilde{\chi}(\gamma,\gamma)/2)[-\tilde{\chi}(\gamma,\gamma)]\right)^{\vee}, 
\]
the dual of the equivariant compactly supported critical cohomology with coefficients in the sheaf of vanishing cycles, tensored with a power of the Tate motive, as in \cite{COHA}.  To keep things relatively simple we can work with $\Gr^{\Wt}(\mathcal{H}_{\tilde{Q},W,\gamma})$, the associated graded object with respect to the weight filtration of mixed Hodge structures.  The space $\Gr^{\Wt}(\mathcal{H}_{\tilde{Q},W})=\bigoplus_{\gamma}\Gr^{\Wt}(\mathcal{H}_{\tilde{Q},W,\gamma})$ is an object in the category of $\mathbb{Z}_{\geq 0}^{Q_0}\oplus\mathbb{Z}_{\Wt}\oplus\mathbb{Z}_{\SC}$-graded $\mathbb{Q}$ vector spaces, and $\mathbb{Q}(n/2)[n]$ may be treated as a 1-dimensional vector space concentrated in degree $(0,-n,-n)\in\mathbb{Z}_{\geq 0}^{Q_0}\oplus\mathbb{Z}_{\Wt}\oplus\mathbb{Z}_{\SC}$.  Here $\mathbb{Z}_{\Wt}\cong\mathbb{Z}_{\SC}\cong\mathbb{Z}$ --- the subscripts are there to serve as a reminder that the first copy of $\mathbb{Z}$ is keeping track of the grading induced by weights of mixed Hodge structures, and the second is keeping track of the cohomological degree.

Similarly we define
\[
\mathcal{H}^{\Sp}_{\tilde{Q},W,\gamma}:=\left(\Ho_{c,G_{\gamma}}\left(M_{\tilde{Q},\gamma},\phi_{\tr(W)_{\gamma}|_{M^{\Sp}_{\tilde{Q},\gamma}}}\right)\otimes \mathbb{Q}(-\tilde{\chi}(\gamma,\gamma)/2)[-\tilde{\chi}(\gamma,\gamma)]\right)^{\vee}.  
\]
%We define the Hall algebra multiplications $\mathcal{H}_{\tilde{Q},W}\otimes \mathcal{H}_{\tilde{Q},W}\rightarrow \mathcal{H}_{\tilde{Q},W}$ and $\mathcal{H}^{\Sp}_{\tilde{Q},W}\otimes \mathcal{H}^{\Sp}_{\tilde{Q},W}\rightarrow \mathcal{H}^{\Sp}_{\tilde{Q},W}$as in \cite{COHA, Chicago2}.  Each degree $i$ piece with respect to the cohomological grading $\mathcal{H}_{\tilde{Q},W,\gamma}$ carries a mixed Hodge structure, and the Hall algebra operations are morphisms of mixed Hodge structures, which is why $\Gr_{\Wt}(\mathcal{H}_{\tilde{Q},W})$ and $\Gr_{\Wt}(\mathcal{H}^{\Sp}_{\tilde{Q},W})$ can be considered as $\mathbb{Z}^{Q_0}\oplus\mathbb{Z}_{\SC}\oplus\mathbb{Z}_{\Wt}$-graded algebras.
In Theorems \ref{Chi2} and \ref{dimred} we recall two fundamental results regarding the mixed Hodge structures $\mathcal{H}_{\tilde{Q},W}$ and $\mathcal{H}_{\tilde{Q},W}^{\Sp}$.
\begin{theorem}\cite[Thm.A, Sec.6.3]{DaMe15b}
\label{Chi2}
There exists a sub $(\mathbb{Z}_{\geq 0}^{Q_0}\setminus\{0\})\oplus\mathbb{Z}_{\Wt}\oplus\mathbb{Z}_{\SC}$-graded vector space $\Gr^{\Wt}(\mathfrak{g}^{\nilp}_{\prim})\subset \Gr^{\Wt}(\mathcal{H}^{\nilp}_{\tilde{Q},W})$ such that there is an isomorphism of $\mathbb{Z}_{\geq 0}^{Q_0}\oplus\mathbb{Z}_{\Wt}\oplus\mathbb{Z}_{\SC}$-graded vector spaces
\begin{equation}
\label{Cinte}
\Sym\left(\Gr^{\Wt}\left(\mathfrak{g}^{\nilp}_{\prim}[u]\right)\right)\cong \Gr^{\Wt}\left(\mathcal{H}^{\nilp}_{\tilde{Q},W}\right),
\end{equation}
with $u$ a formal variable of $\mathbb{Z}_{\geq 0}^{Q_0}\oplus \mathbb{Z}_{\Wt}\oplus\mathbb{Z}_{\SC}$-degree $(0,2,2)$.
\end{theorem}
Here we use the Koszul sign rule, with respect to the cohomological grading.  So if $v$ and $v'$ are homogeneous of $\mathbb{Z}_{\geq 0}^{Q_0}\oplus\mathbb{Z}_{\Wt}\oplus\mathbb{Z}_{\SC}$-degree $(\gamma,n,j)$ and $(\gamma',n',j')$ respectively, supercommutativity means that there is an equality
\[
v\cdot v'=(-1)^{jj'}v'\cdot v.
\]
\begin{remark}
Theorem \ref{Chi2} is a special case of the `restricted' version of \cite[Thm.A]{DaMe15b} described in \cite[Sec.6.3]{DaMe15b}.  In fact the mixed Hodge structure $\mathfrak{g}^{\nilp}_{\prim,\gamma}$ is explicitly defined there; we briefly recall the construction.  Let $\mathcal{M}^{0}_{\gamma}$ be the coarse moduli space of semisimple representations of $\tilde{Q}$.  Explicitly, $\mathcal{M}^{0}_{\gamma}$ is the spectrum of the ring of $G_{\gamma}$-invariant functions on the space $M_{\tilde{Q},\gamma}$.  There is a subspace $\Msp^{0,\nilp}_{\gamma}\subset\Msp^{0}_{\gamma}$ defined by the equations $\tr(\omega_i^n)=0$ for all $i\in Q_0$ and $n\geq 1$.  If a simple $\mathbb{C}\tilde{Q}$-module of dimension $\gamma$ exists, define the mixed Hodge module 
\[
\IC_{\gamma}:=\IC_{\Msp^0_{\gamma}}(\mathbb{Q})\otimes\mathbb{Q}(-\tilde{\chi}(\gamma,\gamma)/2)[-\tilde{\chi}(\gamma,\gamma)], 
\]
otherwise define $\IC_{\gamma}:=0$.  The potential $W$ defines a function $\tr(W)_{\gamma}$ on the space $\mathcal{M}^0_{\gamma}$, and we define 
\[
\mathfrak{g}^{\nilp}_{\prim,\gamma}=\Ho_c\left(\mathcal{M}_{\gamma}^0,(\phi_{\tr(W)_{\gamma}}\IC_{\gamma})|_{\mathcal{M}_{\gamma}^{0,\nilp}}\right)^{\vee}.
\]
\end{remark}
\begin{remark}
Theorem A from \cite{DaMe15b} is a statement inside the category of $\mathbb{Z}_{\geq 0}^{Q_0}$-graded \textit{monodromic} mixed Hodge structures, and as such uses unpublished work of Saito, proving a Thom--Sebastiani type theorem for such mixed Hodge structures.  In the case considered in this paper, however, we may use Theorem \ref{dimred} to identify the monodromic mixed Hodge structure $\mathcal{H}^{\nilp}_{\tilde{Q},W}$ with an element of the full subcategory of monodromic mixed Hodge structures containing ordinary mixed Hodge structures (informally, mixed Hodge structures equipped with trivial monodromy), and Theorem \ref{Chi2} becomes a less technical statement inside the symmetric monoidal category of $\mathbb{Z}_{\geq 0}^{Q_0}$-graded mixed Hodge structures (without monodromy).  The analogue of the Thom--Sebastiani isomorphism in this non-monodromic setting is the Kunneth isomorphism --- see \cite[Prop.A.8]{Chicago2} for a precise statement.
\end{remark}
%This is a special case of the main theorem of \cite{Chicago2}, using that $(\tilde{Q},W)$ has a self-duality structure, by Proposition \ref{sduality}.
\begin{theorem}\cite[Cor.A.9]{Chicago2}
\label{dimred}
There are isomorphisms of cohomologically graded mixed Hodge structures 
\begin{align}
\label{nnilpver}
\mathcal{H}_{\tilde{Q},W,\gamma}\cong&\left(\Ho_{c,G_{\gamma}}(\overline{Z}_{\tilde{Q},W,S,\gamma},\mathbb{Q})\otimes\mathbb{Q}(-l/2)[-l]\right)^{\vee} \cong \left(\Ho_{c,G_{\gamma}}(Z_{\tilde{Q},W,S,\gamma},\mathbb{Q})\otimes\mathbb{Q}(-l'/2)[-l']\right)^{\vee}\\
\label{nilpver}
\mathcal{H}^{\Sp}_{\tilde{Q},W,\gamma}\cong&\left(\Ho_{c,G_{\gamma}}(\overline{Z}^{\nilp}_{\tilde{Q},W,S,\gamma},\mathbb{Q})\otimes\mathbb{Q}(-l/2)[-l]\right)^{\vee}\cong\left(\Ho_{c,G_{\gamma}}(Z^{\nilp}_{\tilde{Q},W,S,\gamma},\mathbb{Q})\otimes\mathbb{Q}(-l'/2)[-l']\right)^{\vee}
\end{align}
where 
\[
l=\tilde{\chi}(\gamma,\gamma)
\]
and
\[
l'=\tilde{\chi}(\gamma,\gamma)+2\sum_{a\in S}\gamma(s(a))\gamma(t(a)).
\]
\end{theorem}
This is a special case of \cite[Cor.A.9]{Chicago2}, which applies directly to (\ref{nnilpver}), and applies to (\ref{nilpver}) on account of Proposition \ref{stableunderpi}.  The extra shift in $l'$ comes from the fact that the relative (complex) dimension of $\pi$ (from (\ref{pidef})) is $\sum_{a\in S}\gamma(s(a))\gamma(t(a))$.  On the other hand, for all dimension vectors $\gamma\in \mathbb{Z}^{Q_0}$,
\begin{equation}
\tilde{\chi}(\gamma,\gamma)+2\sum_{a\in S}\gamma(s(a))\gamma(t(a))=0,
\end{equation}
and so we deduce that $\mathcal{H}_{\tilde{Q},W,\gamma}$ and $\mathcal{H}^{\Sp}_{\tilde{Q},W,\gamma}$ are given by the \textit{unshifted}, \textit{untwisted} equivariant compactly supported cohomology of the spaces $Z_{\tilde{Q},W,S,\gamma}$ and $Z^{\nilp}_{\tilde{Q},W,S,\gamma}$, i.e. we have isomorphisms of mixed Hodge structures
\begin{align}
\mathcal{H}_{\tilde{Q},W,\gamma}\cong& \left(\Ho_{c,G_{\gamma}}(Z_{\tilde{Q},W,S,\gamma},\mathbb{Q})\right)^{\vee}\\ \label{nilpisoo}
\mathcal{H}^{\Sp}_{\tilde{Q},W,\gamma}\cong& \left(\Ho_{c,G_{\gamma}}(Z^{\nilp}_{\tilde{Q},W,S,\gamma},\mathbb{Q})\right)^{\vee}.
\end{align}

\section{Purity}
Let $L$ be an object in the derived category of mixed Hodge structures.  We say that $L$ is \textit{pure} if the $j$th cohomology of $L$ is pure of weight $j$.  We begin this section with a conjecture\footnote{While this paper was being edited, this conjecture was proved, again via cohomological Donaldson--Thomas theory \cite{IntPre16}.}.
\begin{conjecture}
\label{PurityConj}
For arbitrary $Q$ and $\gamma\in\mathbb{Z}_{\geq 0}^{Q_0}$, the mixed Hodge structure on $\Ho_{c,G_{\gamma}}(Z_{\tilde{Q},W,S,\gamma},\mathbb{Q})$ is pure.
\end{conjecture}
\begin{example}
Assume that $Q$ has no arrows.  Then $(\tilde{Q},W)$ is a quiver with the potential $W=0$.  This implies that each space $\mathcal{H}_{\tilde{Q},W,\gamma}$ carries a pure mixed Hodge structure, since it is the $G_{\gamma}$-equivariant compactly supported cohomology of the affine space $\AA^{\sum_{i\in Q_0}\gamma(i)^2}$, with the trivial shift, as in this case $\tilde{\chi}=0$.
\end{example}
\begin{proposition}
\label{purepaving}
Let $G$ be an algebraic group, and let $X$ be a $G$-equivariant variety, such that $X$ admits a stratification $\emptyset=X_0\subset X_1\subset\ldots X_t=X$ by $G$-invariant subvarieties, in the sense that each $Y_r:=X_r\setminus X_{r-1}$ is open in $X_r$.  Assume that for each $r$, there is an algebraic subgroup $N_r\subset G$ and an inclusion $g_r\colon \mathbb{A}_{\Cp}^{s_r}\rightarrow Y_r$ of an $N_r$-invariant subspace such that the morphism of stacks $g'_r\colon [\mathbb{A}_{\Cp}^{s_r}/N_r]\rightarrow [Y_r/G]$ is an isomorphism.  We assume further that the equivariant cohomology $\Ho_{c,N_r}(\pt,\mathbb{Q})$ is pure.  Then $\Ho_{c,G}(X,\mathbb{Q})$ is pure.
\end{proposition}
\begin{proof}
The isomorphism $g'_r$ induces an isomorphism in compactly supported cohomology 
\begin{equation}
\label{purepar}
\Ho_{c,N_r}(\mathbb{A}_{\Cp}^{s_r},\mathbb{Q})\rightarrow \Ho_{c,G}(Y_r,\mathbb{Q}).
\end{equation}
There is a Gysin isomorphism $\Ho_{c,N_r}^{j-2s_r}(\pt,\mathbb{Q})\rightarrow \Ho_{c,N_r}^j(\mathbb{A}_{\Cp}^{s_r},\mathbb{Q})$ which shifts weights by $2s_r$, from which we deduce that the left hand side of (\ref{purepar}) is pure, and so the right hand side is too.  The proposition then follows from the fact that the long exact sequence in compactly supported cohomology
\[
\rightarrow \Ho_{c,G}^j(Y_r,\mathbb{Q})\rightarrow \Ho_{c,G}^j(X_{r},\mathbb{Q})\rightarrow \Ho_{c,G}^j(X_{r-1},\mathbb{Q})\rightarrow
\]
is a complex in the category of mixed Hodge structures, and induction on $r$.
\end{proof}
\begin{theorem}
\label{purity_thm}
For all $\gamma\in\mathbb{Z}_{\geq 0}^{Q_0}$ the mixed Hodge structure on $\Ho_{c,G_{\gamma}}(Z^{\nilp}_{\tilde{Q},W,S,\gamma},\mathbb{Q})$ is pure.
\end{theorem}
\begin{proof}
Since the left hand side of (\ref{nilpver}) does not depend on $S$ we may pick whichever cut $S$ we like, as long as it contains none of the loops $\omega_i$, so that we may use Proposition \ref{stableunderpi} to prove (\ref{nilpver}) in the first place.  We assume that $S$ consists of the arrows $\tilde{a}$ for $a\in Q_1$.  We consider a different description of the representation varieties $Z_{\tilde{Q},W,S,\gamma}$.  For $\tilde{a}\in S$ we have 
\[
\frac{\partial W}{\partial \tilde{a}}=\omega_{t(a)}a-a\omega_{s(a)},
\]
%\begin{align}
%\label{newdesc}
%\frac{\partial \tilde{W}}{\partial b}=\begin{cases} a^*\omega_{t(a)}-\omega_{s(a)}a^*&\mathrm{if}\hbox{ }b=a\\ \omega_{t(a)}a-a\omega_{s(a)}&\mathrm{otherwise.}\end{cases} 
%\end{align}
and we deduce that $Z_{\tilde{Q},W,S,\gamma}$ is the space of pairs $(\rho,\lambda)$, where 
\[
\rho\in \prod_{b\in Q_1}\Hom(\Cp^{\gamma(s(b))},\Cp^{\gamma(t(b))})=M_{Q,\gamma}
\]
and 
\[
\lambda\in\prod_{i\in Q_0}\End(\Cp^{\gamma(i)})
\]
is an endomorphism of $\rho$ considered as a representation of $Q$.  Similarly, $Z^{\nilp}_{\tilde{Q},W,S,\gamma}$ is the space of such pairs $(\rho,\lambda)$, where $\lambda$ is a nilpotent endomorphism of $\rho$.

A nilpotent endomorphism of the space $\Cp^r$ is defined, up to conjugation, by its Jordan normal form, which is defined in turn by the partition of $r$ induced by taking the sizes of Jordan normal blocks.  We let $\mathcal{P}(\gamma)$ be the set of partitions of $\gamma$, i.e. the set of assignments of partitions $\gamma(i)=\pi_1(i)+\ldots+\pi_{k_i}(i)$, where $\pi_s(i)\geq \pi_{s+1}(i)>0$ for all $s$, to each of the entries $\gamma(i)$.  We let $\mathcal{P}$ be the union of the $\mathcal{P}(\gamma)$ for $\gamma\in\mathbb{Z}^{Q_0}_{\geq 0}$.  Fixing the dimension vector $\gamma$, the set of choices of $\lambda$, up to the action of the gauge group $G_{\gamma}$, is in natural correspondence with $\mathcal{P}(\gamma)$.  The centralizer $N_{\pi}$ of such an $\lambda$ is then an affine bundle over a product of linear groups, and in particular has pure $\Ho_{c,N_{\pi}}(\pt,\mathbb{Q})$ and is special, in the sense that principle $N_{\pi}$-bundles are Zariski locally trivial; we will more explicitly describe the $N_{\pi}$ in the next section, with Lemma \ref{NCalc}.

For $\pi\in\mathcal{P}$ we define $Y_{\pi}$ to be the space of pairs $(\rho,\lambda)$ such that $\lambda$ belongs to the conjugacy class corresponding to $\pi$.  We define a partial ordering on the set $\mathcal{P}(\gamma)$ by the prescription that $\pi>\pi'$ if the orbit of the nilpotent cone $\mathcal{N}_{\gamma}\subset \prod_{i\in Q_0} \mathrm{End}(\mathbb{C}^{\gamma(i)})$ corresponding to $\pi$ contains the orbit corresponding to $\pi'$ in its closure.  Since $\mathcal{P}(\gamma)$ is finite we can complete this partial ordering to a total ordering.  We pick such a total ordering, and we define 
\[
Z^{\pi}_{\gamma}=\bigcup_{\pi'\leq \pi} Y_{\pi'},
\]
giving a stratification of $Z^{\nilp}_{\tilde{Q},W,S,\gamma}$ (in the sense of Proposition \ref{purepaving}).

Given a multipartition $\pi\in\mathcal{P}(\gamma)$, we write $\pi(i)=(1^{\Psi_{i,1}},2^{\Psi_{i,2}},\ldots)$.  Given a pair $(\rho,\lambda)$ belonging to the space $Y_{\pi}$, we consider it as a representation of the quiver $Q$ in the category of nilpotent $\Cp[t]$-modules, where each $i\in Q_0$ is sent to $\bigoplus_{\alpha\in\mathbb{Z}_{\geq 1}}(\Cp[t]/t^{\alpha})^{\oplus^{\Psi_{i,\alpha}}}$.  Given two numbers $\alpha,\alpha'\in\mathbb{Z}_{\geq 1}$, there is an isomorphism $\Hom(\Cp[t]/t^{\alpha},\Cp[t]/t^{\alpha'})\cong\mathbb{A}_{\Cp}^{\min(\alpha,\alpha')}$, and so we deduce that 
\[
[Y_{\pi}/G_{\gamma}]\cong\left[\left(\prod_{a\in Q_1,\alpha,\alpha'\in\mathbb{Z}_{\geq 1}} \mathbb{A}_{\Cp}^{\min(\alpha,\alpha')\Psi_{s(a),\alpha}\Psi_{t(a),\alpha'}}\right)/N_{\pi}\right].
\]
The theorem then follows from Proposition \ref{purepaving}.
\end{proof}
Consider the partition function 
\begin{equation}
\label{genfun}
\chi_{q,x}\left(\Gr^{\Wt}\left(\mathcal{H}^{\nilp}_{\tilde{Q},W}\right)\right):=\sum_{(\gamma,n,j)\in\mathbb{Z}_{\geq 0}^{Q_0}\oplus\mathbb{Z}_{\Wt}\oplus\mathbb{Z}_{\SC}}(-1)^j\dim\left(\Gr^{\Wt}_n\left( \mathcal{H}^{\Sp,j}_{\tilde{Q},W,\gamma}\right)\right)x^{\gamma}q^{n/2}\in\mathbb{Z}((q^{ 1/2}))[[x_i| i\in {Q_0}]]
\end{equation}
where we use the standard notation $x^{\gamma}:=\prod_{i\in Q_0}x_i^{\gamma(i)}$.  Fixing $\gamma \in\mathbb{Z}_{\geq 0}^{Q_0}$, the coefficient of $x^{\gamma}$ is given by the weight polynomial\footnote{Strictly speaking, this should be called the weight formal power series.} of $\Ho_{c,G_{\gamma}}(M^{\Sp}_{\tilde{Q},\gamma},\phi_{\tr(W)_{\gamma}})^{\vee}$, and in particular the infinite alternating sum given by fixing a power of $q^{1/2}$ and $x$ is well-defined, and so (\ref{genfun}) is well-defined --- in fact by Theorem \ref{dimred} and Theorem \ref{purity_thm} all these mixed Hodge structures are pure and we deduce
\begin{align*}
\chi_{q,x}\left(\Gr^{\Wt}\left(\mathcal{H}^{\nilp}_{\tilde{Q},W}\right)\right)=&\sum_{(\gamma,n)\in\mathbb{Z}_{\geq 0}^{Q_0}\oplus\mathbb{Z}}(-1)^n\dim\left( \mathcal{H}^{\Sp,n}_{\tilde{Q},W,\gamma}\right)x^{\gamma}q^{n/2}
\\=&\sum_{(\gamma,n)\in\mathbb{Z}_{\geq 0}^{Q_0}\oplus\mathbb{Z}}\dim \left(\mathcal{H}^{\Sp,2n}_{\tilde{Q},W,\gamma}\right)x^{\gamma}q^{n}.
\end{align*}
By (\ref{nilpisoo}) and the proof of Theorem \ref{purity_thm} the terms in the first equality, for odd $n$, are all zero, which accounts for the second equality.

\begin{definition}
\label{SymDef}
Given an element of the additive group
\begin{align}
\label{Rplus}
\RR:=\left\{f=\sum_{\gamma\in\mathbb{Z}^{Q_0}_{\geq 0}}f_{\gamma}(q^{1/2})x^{\gamma}\in \mathbb{Z}((q^{ 1/2}))[[x_i|i\in {Q_0}]]\textrm{ such that } f_{0}=0\right\}
\end{align}
we write $f_{\gamma}(q^{1/2})=\sum_{n\in\mathbb{Z}}f_{\gamma,n}q^{n/2}$ and we define
\[
\KSym (f)=\sum_{h}\prod_{(\gamma,n)\in\mathbb{Z}_{\geq 0}^{Q_0}\oplus\mathbb{Z}}\binom{f_{\gamma,n}+h(\gamma,n)-1}{h(\gamma,n)}x^{h(\gamma,n)\gamma}q^{h(\gamma,n)n/2},
\]
where the sum is over functions $h\colon \mathbb{Z}_{\geq 0}^{Q_0}\oplus\mathbb{Z}\rightarrow \mathbb{Z}_{\geq 0}$ sending all but finitely many elements to zero, and by convention, the term in the sum corresponding to $h=0$ is $1$.  
\end{definition}
This defines a group homomorphism 
\[
\KSym \colon \RR\rightarrow 1+\RR
\]
where $1+\RR$ is considered as a subgroup of the group of units of $\mathbb{Z}((q^{1/2}))[[x_i|i\in {Q_0}]]$.  There is an inverse to the function $\KSym $, and so it is injective.  The function $\KSym $ is known as the plethystic exponential.  For later purposes we define the group 
\begin{align}
\label{Rminus}
\RR_-:=\left\{f=\sum_{\gamma \in\mathbb{Z}^{Q_0}_{\geq 0}}f_{\gamma}(q^{1/2})x^{\gamma}\in \mathbb{Z}((q^{-1/2}))[[x_i|i\in {Q_0}]]\textrm{ such that } f_{0}=0\right\},
\end{align}
and define the plethystic exponential similarly
\[
\KSymm \colon \RR_-\rightarrow 1+\RR_-.
\]

If $X$ is a variety acted on by an algebraic group $G$, then the cohomology $\Ho_G(X,\mathbb{Q})$ need not be bounded, and $\chi_q(\Ho_G(X,\mathbb{Q}))\in \mathbb{Z}((q^{1/2}))$.  Similarly, $\chi_q(\Ho_{c,G}(X,\mathbb{Q}))\in\mathbb{Z}((q^{-1/2}))$.  Accordingly, $1+\RR$ is the natural target for weight polynomials applied to $\mathbb{Z}_{\geq 0}^{Q_0}$-graded equivariant cohomology, while $1+\RR_-$ is the natural target for weight polynomials applied to $\mathbb{Z}_{\geq 0}^{Q_0}$-graded compactly supported equivariant cohomology.

The function $\KSym $ has a more illuminating description, which also explains the link with free supercommutative algebras.  Let $V$ be a $\mathbb{Z}_{\geq 0}^{Q_0}\oplus\mathbb{Z}_{\Wt}\oplus \mathbb{Z}_{\SC}$-graded vector space, and let $f$ be its characteristic function
\[
f=\sum_{(\gamma,n,j)\in \mathbb{Z}_{\geq 0}^{Q_0}\oplus\mathbb{Z}_{\Wt}\oplus\mathbb{Z}_{\SC}}(-1)^j\dim(V_{\gamma,n}^j)x^{\gamma}q^{n/2},
\]
which we assume is well-defined and belongs to $\RR$.  Then $\KSym (f)$ is the characteristic function of the underlying graded vector space of the free supercommutative algebra generated by $V$.  In other words, $\KSym $ is the decategorification of the operation of forming the free supercommutative algebra, explaining our choice of notation in place of, say, $\mathbf{EXP}$.

In terms of this operation, we may write, using the cohomological integrality theorem (Theorem \ref{Chi2}):
\[
\chi_{q,x}\left(\Gr^{\Wt}\left(\mathcal{H}^{\nilp}_{\tilde{Q},W}\right)\right)=\KSym \left(\sum_{(\gamma,n,j)\in(\mathbb{Z}_{\geq 0}^{Q_0}\setminus \{0\})\oplus\mathbb{Z}_{\Wt}\oplus\mathbb{Z}_{\SC}}\dim\left(\Gr^{\Wt}_n(\mathfrak{g}^{\nilp,j}_{\prim,\gamma})\right)(-1)^{j}x^{\gamma}q^{n/2}(1-q)^{-1}\right).
\]
Here the $(1-q)^{-1}$ factor is the contribution of $\chi_q(\Ho_{\mathbb{C}^*}(\pt))=(1-q)^{-1}$ to the tensor product $\mathfrak{g}^{\nilp}_{\prim,\gamma}[u]\cong \mathfrak{g}^{\nilp}_{\prim,\gamma}\otimes \Ho_{\mathbb{C}^*}(\pt)$ on the left hand side of Equation (\ref{Cinte}).  The following is trivial.
\begin{lemma}
\label{Lsimp}
If $A$ is a free supercommutative $\mathbb{Z}\oplus\mathbb{Z}_{\SC}$-graded algebra generated by the graded subspace $V$, and $A_{m}^n=0$ unless $m=n$, then $V_{m}^n=0$ unless $m=n$.
\end{lemma}
We deduce from Theorem \ref{purity_thm} and Lemma \ref{Lsimp} that we may write 
\begin{equation}
\label{alignment}
\chi_{q,x}\left(\Gr^{\Wt}\left(\mathcal{H}^{\nilp}_{\tilde{Q},W}\right)\right)=\KSym \left(\sum_{(\gamma,m)\in(\mathbb{Z}_{\geq 0}^{Q_0}\setminus\{0\})\oplus\mathbb{Z}}\dim(\mathfrak{g}^{\nilp,m}_{\prim,\gamma})x^{\gamma}(-q^{1/2})^{m}(1-q)^{-1}\right),
\end{equation}
i.e. $\Gr^{\Wt}_n(\mathfrak{g}^{\nilp,j}_{\prim,\gamma})=0$ for $j\neq n$.  In other words, $\mathfrak{g}_{\prim}^{\nilp}$ is a pure $(\mathbb{Z}_{\geq 0}^{Q_0}\setminus \{0\})$-graded mixed Hodge structure.
\section{Kac polynomials}
\label{ksection}
Associated to the quiver $Q$ and a dimension vector $\gamma\in\mathbb{Z}_{\geq 0}^{Q_0}$ is the \textit{Kac polynomial} $a_{\gamma}(q)$ introduced by Kac \cite{kac83}, counting the number of isomorphism classes of absolutely indecomposable $\gamma$-dimensional representations of $Q$ over $\mathbb{F}_q$.

We recall a theorem of Hua, in terms of the $\KSym $ operation of Definition \ref{SymDef}.  Firstly we will need some notation.  For two partitions $\pi$, $\pi'$ we define $\langle \pi,\pi'\rangle:=\sum_{\alpha\in\mathbb{Z}_{\geq 1}}\pi_{\alpha}\pi'_{\alpha}$.  If we write a partition $\pi$ in the notation $\pi=(1^{\psi_1},2^{\psi_2},\ldots)$, we define $b_{\pi}(q)=\prod_{j}(1-q)\ldots(1-q^{\psi_j})$.  If $\pi\in\mathcal{P}$, then $\pi\in\mathcal{P}(\gamma)$ for some $\gamma$ and we define $|\pi|:=\gamma$.
\begin{theorem}\cite[Thm.4.9]{Hua00}
\label{HuaThm}
There is an equality of generating functions
\begin{equation}
\label{HuaFormula}
\KSym \left(\sum_{\substack{\gamma\in\mathbb{Z}_{\geq 0}^{Q_0},\\ \gamma\neq 0}} -x^{\gamma}a_{\gamma}(q)(1-q)^{-1}\right)=\sum_{\pi\in\mathcal{P}}c_{\pi}(q)x^{|\pi|}\in 1+\mathcal{R},
\end{equation}
where 
\[
c_{\pi}(q):=\frac{\prod_{a\in Q_1}q^{\langle \pi(s(a)),\pi(t(a))\rangle}}{\prod_{i\in Q_0}q^{\langle \pi(i),\pi(i)\rangle}b_{\pi(i)}(q^{-1})}.
\]
\end{theorem}

\begin{lemma}\cite[Lem.3.1]{Hua00}
\label{HuaLemma}
Let $\pi$ and $\pi'$ be two partitions, which we write in the notation $\pi=(1^{\psi_1},2^{\psi_2},\ldots)$.  Then there is an identity
\[
\langle \pi,\pi'\rangle=\sum_{\alpha,\alpha'\in\mathbb{Z}_{\geq 1}}\min(\alpha,\alpha')\psi_{\alpha}\psi'_{\alpha'}.
\]
\end{lemma}
By (\ref{nilpisoo}) there is an isomorphism $\mathcal{H}^{\Sp,\vee}_{\tilde{Q},W,\gamma}\cong \Ho_{c,G_{\gamma}}(Z^{\nilp}_{\tilde{Q},W,S,\gamma},\mathbb{Q})$.  From the proof of Theorem \ref{purity_thm} we may alternatively write, using Lemma \ref{HuaLemma}, and the fact that, as we will prove in Lemma \ref{NCalc}, each $N_{\pi}$ is special,
\begin{equation}
\label{donetop}
\chi_{q,x}\left(\bigoplus_{\gamma\in\mathbb{Z}_{\geq 0}^{Q_0}}\mathcal{H}^{\nilp,\vee}_{Q,W,\gamma}\right)=\sum_{\pi\in\mathcal{P}}\frac{\prod_{a\in Q_1}q^{\langle \pi(s(a)),\pi(t(a))\rangle}}{\chi_q(\Ho_c(N_{\pi},\mathbb{Q}))}x^{|\pi|}\in 1+\RR_-
\end{equation}
where we expand the right hand side in powers of $q^{-1}$.  In Equation (\ref{donetop}) the denominator $\chi_q(\Ho_c(N_{\pi},\mathbb{Q}))$ is the weight polynomial of the compactly supported cohomology of the algebraic group $N_{\pi}$, which we now calculate.
\begin{lemma}
\label{NCalc}
There is an equality
\begin{equation}
\label{NCalcEq}
\chi_q\left(\Ho_c(N_{\pi},\mathbb{Q})\right)=\prod_{i\in Q_0}q^{\langle \pi(i),\pi(i)\rangle}b_{\pi(i)}(q^{-1}),
\end{equation}
and $N_{\pi}$ is special.
\end{lemma}
\begin{proof}
It is sufficient to prove the claim under the assumption that $Q$ has only one vertex, which will ease the notation somewhat.  Let $K_{\pi}= \bigoplus_{\alpha\geq 1} (\Cp[t]/t^{\alpha})^{\oplus \psi_{\alpha}}$ be the nilpotent $\Cp[t]$-module corresponding to the partition $\pi=(1^{\psi_1},2^{\psi_2},\ldots)$, then $N_{\pi}=\Aut(K_{\pi})$.  Let $H_{\alpha}$ be the kernel of the map $\Aut(K_{\pi})\rightarrow \Aut\left(K_{\pi}\otimes_{\mathbb{C}[t]} \Cp[t]/t^{\alpha}\right)$.  Then as a complex algebraic group, $N_{\pi}/H_1$ is isomorphic to the parabolic subgroup $P$ with Levi subgroup $L:=\prod_{\alpha\in\mathbb{Z}_{\geq 1}}\Gl_{\psi_{\alpha}}(\Cp)$.  We define $H_0$ be the kernel of the composition of natural surjections
\[
N_{\pi}\rightarrow P\rightarrow L.
\]
For $\alpha\geq 1$, each of the unipotent algebraic groups $H_{\alpha-1}/H_{\alpha}$ is isomorphic to a vector space, as a complex variety.  There is an isomorphism $H_0\cong \prod_{\alpha,\alpha'\in\mathbb{Z}_{\geq 1}} \mathbb{A}^{(\min(\alpha,\alpha')-\delta_{\alpha,\alpha'})\psi_{\alpha}\psi_{\alpha'}}$ as complex varieties, where $\delta_{\alpha,\alpha'}$ is the Kronecker delta function.  Since the class of special groups is closed under extensions and contains general linear groups and $\mathbb{C}$ (see the discussion after \cite[Def.2.1]{Jo07b}), it follows that $N_{\pi}$ is special, and (\ref{NCalcEq}) follows from the equations
\begin{align*}
\chi_q\left(\Ho_c(\Aut(K_{\pi})/H_0,\mathbb{Q})\right)=&b_{\pi}(q^{-1})q^{\sum_{\alpha} \psi_{\alpha}^2}\\
\chi_q(\Ho_c(H_0,\mathbb{Q}))= &q^{\langle \pi,\pi\rangle-\sum_{\alpha} \psi_{\alpha}^2}.
\end{align*}
\end{proof}
We deduce from (\ref{donetop}) and Lemma \ref{NCalc} that 
\begin{equation}
\label{nHua}
\chi_{q,x}\left(\bigoplus_{\gamma\in\mathbb{Z}_{\geq 0}^{Q_0}}\mathcal{H}^{\nilp,\vee}_{Q,W,\gamma}\right)=\sum_{\pi\in\mathcal{P}}c_{\pi}(q)x^{|\pi|}\in 1+\mathcal{R}_-.
\end{equation}
where the right hand side of (\ref{nHua}) becomes the right hand side of (\ref{HuaFormula}) after we write them both as power series in $x_i$ with coefficients given by rational functions in $q$.  To compare the left hand sides of (\ref{nHua}) and (\ref{HuaFormula}) we use the following technical lemma\footnote{Thanks go to Sven Meinhardt for originally pointing out this neat fact to me.}.
\begin{lemma}
\label{Slemma}
Let 
\[
f=\sum_{\substack{\gamma\in\mathbb{Z}_{\geq 0}^{Q_0},\\ \gamma\neq 0}}f_{\gamma}(q^{1/2})q^{1/2}(1-q)^{-1}x^{\gamma}
\]
with each $f_{\gamma}(q^{1/2})$ a Laurent polynomial in $q^{1/2}$.  We denote by $f^+$ the formal power series expansion of $f$, considered as an element of $\RR$ as defined in (\ref{Rplus}), and by $f^-$ the power series expansion in $\RR_-$, as defined in (\ref{Rminus}).  Then we can write 
\[
\KSymp(f^{\pm})=\sum_{\gamma\in \mathbb{Z}_{\geq 0}^{Q_0}}r^{\pm}_{\gamma}(q)x^{\gamma}\in 1+\RR_{\pm}, 
\]
where $r^{\pm}_{\gamma}(q^{1/2})$ are rational functions in $q^{1/2}$, and $\KSym_+$ denotes $\KSym$.  Moreover, there is an equality of rational functions $r_{\gamma}^+(q^{1/2})=r_{\gamma}^-(q^{1/2})$.
\end{lemma} 
\begin{proof}
We start by recalling two Euler product expansions of quantum dilogarithms
\begin{align}\label{fermion}
\KSym (-q^{1/2}(1-q)^{-1}x)=&\sum_{m\geq 0}(-q^{1/2})^{m^2}\left(\prod_{s=1}^{m}(1-q^s)^{-1}\right)x^m\\
\KSym ((1-q)^{-1}x)=&\sum_{m\geq 0}\left(\prod_{s=1}^{m}(1-q^s)^{-1}\right)x^m\label{boson}.
\end{align}

Writing $f_{\gamma}(q^{1/2})=\sum_{n\in\mathbb{Z}} f_{\gamma,n}q^{n/2}$ we have
\[
\KSym\left(f^+\right)=\prod_{(\gamma,n)\in(\mathbb{Z}^{Q_0}_{\geq 0}\setminus \{0\})\oplus \mathbb{Z}} \left(\KSym \left(q^{1/2+n/2}(1-q)^{-1}x^{\gamma}\right)\right)^{f_{\gamma,n}}
\]
and we deduce that it is enough to prove the lemma in the case $f=-q^{n/2}x^{\gamma}$.  Then
\begin{align*}
\KSym \left(-q^{1/2+n/2}(1+q+\ldots)x^{\gamma}\right)=&\sum_{m\geq 0}(-q^{1/2})^{m^2}\prod_{1\leq s\leq m}(1-q^s)^{-1}(x^{\gamma}q^{n/2})^m
\end{align*}
by Equation (\ref{fermion}), from which we deduce that 
\begin{equation}
\label{niceform}
r^+_{m\gamma}(q^{1/2})=(-q^{1/2})^{m^2}q^{mn/2}\prod_{1\leq s\leq m}(1-q^s)^{-1}, 
\end{equation}
so rationality follows.  On the other hand, we may rewrite 
\[
-q^{n/2}x^{\gamma}q^{1/2}(1-q)^{-1}=q^{n/2-1/2}x^{\gamma}(1+q^{-1}+\ldots)=:f^-
\]
and
\[
\KSymm\left(q^{n/2-1/2}x^{\gamma}(1+q^{-1}+\ldots)\right)=\sum_{m\geq 0}\left(\prod_{1\leq s\leq m}(1-q^{-s})^{-1}\right)(x^{\gamma}q^{n/2-1/2})^m
\]
via Equation (\ref{boson}) and so $r^-_{m\gamma}(q^{1/2})=\prod_{1\leq s\leq m}(1-q^{-s})^{-1}q^{-m/2}q^{mn/2}=r^+_{m\gamma}(q^{1/2})$ as required.
\end{proof}
Write $r_{\gamma}=\sum_{\pi\in\mathcal{P}(\gamma)}c_{\pi}(q)$.  By Hua's Theorem \ref{HuaThm} and Lemma \ref{Slemma},
\begin{equation}
\label{serf}
\KSymm \left(\sum_{\substack{\gamma\in\mathbb{Z}_{\geq 0}^{Q_0},\\ \gamma\neq 0}} x^{\gamma}a_{\gamma}(q)q^{-1}(1-q^{-1})^{-1}\right)=\sum_{\gamma\in\mathbb{Z}_{\geq 0}^{Q_0}}r_{\gamma}(q)x^{\gamma}\in 1+\RR_-.
\end{equation}
Replacing $\chi_{q,x}(\mathcal{H}^{\nilp,\vee}_{Q,W})$ with $\chi_{q,x}(\mathcal{H}^{\nilp}_{Q,W})$ corresponds to the substitution $q^{1/2}\mapsto q^{-1/2}$, which also intertwines the operations $\KSym$ and $\KSymm$, and so from Equations (\ref{nHua}) and (\ref{serf}) we deduce
\begin{equation}
\label{fkac}
\KSym \left(\sum_{\substack{\gamma\in\mathbb{Z}_{\geq 0}^{Q_0},\\ \gamma\neq 0}} x^{\gamma}qa_{\gamma}(q^{-1})(1-q)^{-1}\right)=\chi_{q,x}(\mathcal{H}^{\nilp}_{Q,W})\in 1+\RR.
\end{equation}
Equating coefficients with (\ref{alignment}), and using injectivity of $\KSym $, we deduce
\begin{equation}
\label{Kerblau}
qa_{\gamma}(q^{-1})=\sum_{n\in 2\mathbb{Z}}\dim(\mathfrak{g}^{\nilp,n}_{\prim,\gamma})q^{n/2},
\end{equation}
where the right hand sum is over the even numbers, as the left hand side contains no odd powers of $q$.  We recover the theorem of Hausel, Letellier and Rodriguez-Villegas:
\begin{theorem}
The Kac polynomials $a_{\gamma}(q)$ have positive coefficients.
\end{theorem}
\begin{remark}
We also recover a result of Kac, which states that the polynomials $a_{\gamma}(q)$ are independent of the orientation of $Q$.  In our framework, this is explained by the fact that the isomorphism class of the pair $(\Cp \tilde{Q},W)$ is independent of the orientation of $Q$.
\end{remark}
\begin{remark}
Rewriting $qa_{\gamma}(q^{-1})$ in the left hand side of (\ref{Kerblau}) as $-q^{1/2}(-q^{1/2}a_{\gamma}(q^{-1}))$, where the removed $-q^{1/2}$ factor is as in (\ref{shiftexp}), our calculations show that the refined DT invariants for the category of representations of the Jacobi algebra for $(\tilde{Q},W)$, such that $\omega_i$ acts nilpotently for all $i$, are given by $-q^{-1/2}a_{\gamma}(q)$.  The substitution $q^{1/2}\mapsto q^{-1/2}$ is due to the fact that in \cite{KS} the vector dual theory is considered in defining refined DT invariants.  By the usual manipulations involving power structures (see e.g. \cite{GLM04, BBS}), one can show that
\[
\chi_{q,x}\left(\Gr^{\Wt}(\mathcal{H}_{\tilde{Q},W}^{\vee})\right)=\left(\chi_{q,x}\left(\Gr^{\Wt}(\mathcal{H}^{\nilp,\vee}_{\tilde{Q},W})\right)\right)^q.
\]
It follows that the refined DT invariants (in the sense of \cite{KS}) for the category of representations of the Jacobi algebra for $(\tilde{Q},W)$ are given by $-q^{1/2}a_{\gamma}(q)$, recovering \cite[Thm.5.1]{Moz11}.  
\end{remark}
\begin{remark}
By the previous remark, the generating function $\chi_{q,x}(\Gr^{\Wt}(\mathcal{H}_{\tilde{Q},W}))$
contains only integral powers of $q$, and all its coefficients are positive.  This is as implied by Conjecture \ref{PurityConj}.
\end{remark}

\bibliographystyle{amsplain}
\bibliography{SPP}

\providecommand{\bysame}{\leavevmode\hbox to3em{\hrulefill}\thinspace}
\providecommand{\MR}{\relax\ifhmode\unskip\space\fi MR }
\providecommand{\MRhref}[2]{%
  \href{http://www.ams.org/mathscinet-getitem?mr=#1}{#2}
}
\providecommand{\href}[2]{#2}
\begin{thebibliography}{10}

\bibitem{BBS}
K.~Behrend, J.~Bryan, and B.~Szendr\H{o}i, \emph{Motivic degree zero
  {D}onaldson-{T}homas invariants}, Invent. Math. \textbf{192} (2013), no.~1,
  111--160.

\bibitem{IntPre16}
B.~Davison, \emph{{The integrality conjecture and the cohomology of
  preprojective stacks}}, \url{http://arxiv.org/abs/1602.02110}, 2016.

\bibitem{Chicago2}
\bysame, \emph{{The critical CoHA of a quiver with potential}}, Q. J. Math.
  \textbf{0} (2017), 1--69.

\bibitem{DMSS13}
B.~Davison, D.~Maulik, J.~Sch{\"u}rmann, and B.~Szendr\H{o}i, \emph{Purity for
  graded potentials and quantum cluster positivity}, Comp. Math. \textbf{151},
  1913--1944.

\bibitem{DaMe4}
B.~Davison and S.~Meinhardt, \emph{{Donaldson--Thomas theory for categories of
  homological dimension one with potential}},  (2015), arXiv:1512.08898.

\bibitem{DaMe15b}
\bysame, \emph{{Cohomological Donaldson--Thomas theory of a quiver with
  potential and quantum enveloping algebras}},
  \url{http://arxiv.org/abs/math/1601.02479}, Jan 2016.

\bibitem{Efi11}
A.~Efimov, \emph{Quantum cluster variables via vanishing cycles},
  \url{http://arxiv.org/abs/1112.3601}.

\bibitem{GLM04}
S.M. Gusein-Zade, I.~Luengo, and A.~Melle-Hernandez, \emph{{A power structure
  over the Grothendieck ring of varieties}}, Math. Res. Lett. \textbf{11}
  (2004), no.~1, 49--57.

\bibitem{HLRV13}
T.~Hausel, E.~Letellier, and F.~Rodriguez-Villegas, \emph{Positivity for kac
  polynomials and dt-invariants of quivers}, Ann. Math. \textbf{177} (2013),
  1147--1168.

\bibitem{Hua00}
J.~Hua, \emph{{Counting representations of quivers over finite fields}}, J. Alg
  \textbf{226} (2000), 1011--1033.

\bibitem{Jo07b}
D.~Joyce, \emph{Motivic invariants of {A}rtin stacks and stack functions}, Q.
  J. Math \textbf{58} (2007), no.~3, 345--392.

\bibitem{JS08}
D.~Joyce and Y.~Song, \emph{{A theory of generalized Donaldson--Thomas
  invariants}}, Memoires of the AMS (2012).

\bibitem{kac83}
Victor~G Kac, \emph{Root systems, representations of quivers and invariant
  theory}, Invariant theory, Springer, 1983, pp.~74--108.

\bibitem{KS}
M.~Kontsevich and Y.~Soibelman, \emph{{Stability structures, motivic
  Donaldson-Thomas invariants and cluster transformations}},
  \url{http://arxiv.org/abs/0811.2435}, 2008.

\bibitem{COHA}
\bysame, \emph{{Cohomological Hall algebra, exponential Hodge structures and
  motivic Donaldson-Thomas invariants}}, {Comm. Num. Th. and Phys.} \textbf{5}
  (2011), no.~2, 231--252.

\bibitem{Moz11}
S.~Mozgovoy, \emph{{Motivic Donaldson-Thomas invariants and McKay
  correspondence}}, \url{http://arxiv.org/abs/1107.6044}.

\end{thebibliography}

\end{document}